\documentclass[11pt]{nyjm}
\usepackage{mathptmx}
\usepackage{eucal}
\usepackage{amsmath}
\usepackage{amscd}
\usepackage{amssymb}
\usepackage{amsthm}
\usepackage{xspace}
\usepackage[all,tips]{xy}
\usepackage[dvips]{graphicx}
\usepackage{graphics}
\usepackage{epsfig}
\usepackage{verbatim}
\usepackage{syntonly}
\usepackage{hyperref}
\usepackage{amssymb, color, url}
\input xy
\xyoption{all}

\DeclareMathOperator{\GL}{GL}\DeclareMathOperator{\PSL}{PSL}
\DeclareMathOperator{\SL}{SL}

\DeclareMathOperator{\D}{D}

\DeclareMathOperator{\DM}{F}\DeclareMathOperator{\word}{w}
\DeclareMathOperator{\sub}{s}
\DeclareMathOperator{\sol}{sol}
\DeclareMathOperator{\nil}{nil}

\DeclareMathOperator{\rk}{rk}

\DeclareMathOperator{\Z}{\mathbb{Z}}
\DeclareMathOperator{\normF}{F}

\begin{document}
\bibliographystyle{plain}

%-----------------------------------------------------------
%-----------------------------------------------------------

\title{Approximating a group by its solvable quotients}
\author{Khalid Bou-Rabee}
\address{Khalid Bou-Rabee, Department of Mathematics, The University of Michigan, 
2074 East Hall, Ann Arbor, MI 48109-1043}
\email{khalidb@umich.edu}

\newtheorem{theorem}{Theorem}
\newtheorem{claim}{Claim}
\newtheorem{conjecture}[theorem]{Conjecture}
\newtheorem{corollary}[theorem]{Corollary}
\newtheorem{lemma}[theorem]{Lemma}
\newtheorem{question}{Question}
\newtheorem*{remark}{Remark}
\newtheorem{ques}{Question}
\newtheorem*{theoremNL}{Theorem}
\newtheorem{example}{Example}

\newcommand{\innp}[1]{\left< #1 \right>}
\newcommand{\abs}[1]{\left\vert#1\right\vert}
\newcommand{\set}[1]{\left\{#1\right\}}
\newcommand{\norm}[1]{\left\vert \left\vert #1\right\vert\right\vert}
\newcommand{\brac}[1]{\left[#1\right]}
\newcommand{\pr}[1]{\left( #1 \right) }
\newcommand{\su}{\subset}

%-----------------------------------------------------------
% Tasho's macros
%-----------------------------------------------------------

\def\mf#1{\mathfrak{#1}}
\def\mc#1{\mathcal{#1}}
\def\mb#1{\mathbb{#1}}
\def\tx#1{{\rm #1}}
\def\tb#1{\textbf{#1}}
\def\ti#1{\textit{#1}}
\def\ts#1{\textsf{#1}}
\def\tr{\tx{tr}\,}
\def\R{\mathbb{R}}
\def\C{\mathbb{C}}
\def\Q{\mathbb{Q}}
\def\A{\mathbb{A}}
\def\Z{\mathbb{Z}}
\def\N{\mathbb{N}}
\def\F{\mathbb{F}}
\def\lmod{\setminus}
\def\Gal{\rm{Gal}}
\def\ol#1{\overline{#1}}
\def\ul#1{\underline{#1}}
\def\sp{\tx{Spec}}
\def\rk{\tx{rk}}
\def\Ad{\tx{Ad}}
\def\hat{\widehat}
\def\vol{\tx{vol}}
\def\conj#1{\ \underaccent{#1}{\sim}\ }
\def\rw{\rightarrow}
\def\lrw{\longrightarrow}
\def\hrw{\hookrightarrow}
\def\thrw{\twoheadrightarrow}
\def\lw{\leftarrow}
\def\sm{\smallsetminus}

%--
% reid's macros
%--

\def\tr{\mbox{\rm{tr}}}
\def\Hom{\mbox{\rm{Hom}}}
\def\Ram{\mbox{\rm{Ram}}}
\def\PSL{\mbox{\rm{PSL}}}
\def\P{\mbox{\rm{P}}}
\def\d{\mbox{\rm{d}}}
\def\SL{\mbox{\rm{SL}}}
\def\SO{\mbox{\rm{SO}}}
\def\SU{\mbox{\rm{SU}}}
\def\PSU{\mbox{\rm{PSU}}}
\def\PGL{\mbox{\rm{PGL}}}
\def\rk{\mbox{\rm{rk}}}
\def\GL{\mbox{\rm{GL}}}
\def\Gal{\mbox{\rm{Gal}}}
\def\isomorphic{\cong}
\def\A{\cal A}
\def\det{\mbox{\rm{det}}}
\def\cusp{\mbox{\rm{cusp}}}
\def\Vol{\mbox{\rm{Vol}}}
\def\min{\mbox{\rm{min}}}
\def\dim{\mbox{\rm{dim}}}
\def\Stab{\mbox{\rm{Stab}}}
\def\mod{\mbox{mod}}
\def\demo{ {\bf Proof.} }
\def\H{{\bf H}^3}
%\def\qed{ $\sqcup\!\!\!\!\sqcap$}

%-----------------------------------------------------------
%-----------------------------------------------------------

\begin{abstract}
The solvable residual finiteness growth of a group quantifies how well approximated the group is by its finite solvable quotients. In this note we present a new characterization of polycyclic groups which are  virtually nilpotent. That is, we show that a group has solvable residual finiteness growth which is at most polynomial in $\log(n)$ if and only if the group is polycyclic and virtually nilpotent. We also give new results concerning approximating oriented surface groups by nilpotent quotients. As a consequence of this, we prove that a natural number $C$ exists so that any nontrivial element of the $Ck$th term of the lower central series of a finitely generated oriented surface group must have word length at least $k$. Here $C$ depends only on the choice of generating set. Finally, we give some results giving new lower bounds for the solvable residual finiteness growth of some metabelian groups (including the Lamplighter groups).
\end{abstract}
\maketitle

\section*{Introduction}

Let $B_{\Gamma, \mc{X}} (n)$ denote the metric ball of radius $n$ in a group $\Gamma$ generated by a finite set $\mc{X}$ with respect to the word metric $\| \cdot \|_{\Gamma, \mc{X}}$.
In this article we study the following question:

\begin{ques} \label{quantifyresidualfinitenessquestion}
Let $P$ be a property of groups.
How large a finite group with property P do we need to detect elements in $B_{\Gamma, \mc{X}}(n)$? 
That is, what is the smallest integer $\normF^{P}_{\Gamma,X}(n)$ such that each nontrivial element in $B_{\Gamma,X}(n)$ survives through some homomorphism to a group with property $P$ of cardinality no greater than $\normF^{P}_{\Gamma, X}(n)$?
\end{ques}

We will be focusing on two properties: nilpotent ($P= \nil$) and solvable ($P = \sol$).
The asymptotic growth of $F_\Gamma^{\sol}$ is called the \emph{solvable residual finiteness growth} of $\Gamma$, while the asymptotic growth of $F_\Gamma^{\nil}$ is called the \emph{nilpotent residual finiteness growth}.
When the property $P$ is relaxed, we use the notation $\normF_\Gamma$ and call the growth the \emph{normal residual finiteness growth}.

It is known that any virtually nilpotent group has normal residual finiteness growth which is at most polynomial in $\log(n)$ \cite{B10}. Further, any finitely generated linear group with normal residual finiteness growth which is polynomial in $\log(n)$ is virtually nilpotent \cite{BM1}. The author has been unable to find any group that is not  virtually nilpotent with such growth.
Hence, this has lead the author to believe that there may be a positive answer to the following question.

\begin{ques} \label{MainConjecture}
Is it true that if a group $\Gamma$ has normal residual finiteness growth which is at most polynomial in $\log(n)$, then $\Gamma$ is virtually nilpotent?
\end{ques}

Our first result is the following, which resolves the question for a large class of groups. Our proof builds off of the methods in \cite{BM1} (c.f. Theorem 1.1 in that paper).

\begin{theorem} \label{MainTheorem}
If $\Gamma$ has a finite-index subgroup that is residually solvable and finitely generated then the following are equivalent:
\begin{itemize}
\item $\Gamma$ has normal residual finiteness growth which is at most polynomial in $\log(n)$
\item $\Gamma$ is virtually nilpotent.
\end{itemize}
\end{theorem}

\noindent
We also record the following result, which resolves the question for solvable groups, while providing a new characterization of solvable groups which are virtually nilpotent. 
Loosely speaking, the result shows that, except in the ``obvious'' cases, there is a new universal lower bound on the difficulty of detecting elements in finite solvable quotients.

\begin{theorem} \label{SecondaryTheorem}
Let $\Gamma$ be finitely generated. Then the following are equivalent:
\begin{itemize}
\item $\Gamma$ has solvable residual finiteness growth that is polynomial in $\log(n)$ 
\item $\Gamma$ is solvable and virtually nilpotent
\end{itemize}
\end{theorem}

\noindent
The proofs of Theorems \ref{MainTheorem} and \ref{SecondaryTheorem} are in Section \ref{ProofSection}. In our proofs we use in an essential way the results of J.~S.~Wilson \cite{Wi05}.

On the other end of the growth spectrum, the first Grigorchuk group is known to have solvable, normal, and nilpotent residual finiteness growth which is exponential \cite{B10}. Our next main result, proved in Section \ref{freegroupsection}, shows that the fundamental group of an oriented surface has super-polynomial, but not super-exponential, nilpotent residual finiteness growth.
The proof of the upper bound relies on the structure theory of the group $\PSL_2( \Z[i])$, while the lower bound uses a construction due to B.~Bandman, G-M Greuel, F.~Grunewald, B.~Kunyavskii, G.~Pfsiter, and E.~Plotkin \cite{B06}.
 
\begin{theorem} \label{freegrouptheorem}
Let $\Gamma$ be the fundamental group of an oriented surface.
Then
$$
2^n \succeq \normF^{\nil}_{\Gamma} (n) \succeq 2^{\sqrt{n}}.
$$
\end{theorem}
\noindent
From this theorem, we obtain the following corollary (proved in Section \ref{freegroupsection}). Variants of the following corollary for free groups have been shown using different methods (c.f. Theorem 1.2 in \cite{MP}).
\begin{corollary}
Let $\Gamma$ be the fundamental group of an oriented surface with generating set $\mc{X}$.
Then there exists a constant $C > 0$ such that any nontrivial $\gamma \in \Gamma$ satisfies $\gamma \notin \Gamma_{C\| \gamma\|_\mc{X}}$.
\end{corollary}

We conclude, in Section \ref{ExampleSection}, by showing new lower bounds for the normal residual finiteness growth for some wreath products of abelian groups (i.e. the Lamplighter groups). The proofs give explicit constructions of elements in $B_{\Gamma,\mc{X}}(n)$ that are not well-approximated by finite solvable quotients.
These results suggest that a gap might exist for the normal residual finiteness growth of solvable groups that are not polycyclic.

\subsection*{Acknowledgements}
I am especially grateful to my advisor, Benson Farb, and my postdoctoral mentor, Juan Souto, for their endless support and guidance.
I thank Alan Reid for supplying the proof of Claim \ref{ReidClaim}.
I am very grateful to Justin Malestein for his comments on an earlier draft.
And I  thank Ralf Spatzier, Richard Canary, Matthew Stover, and Blair Davey for helpful mathematical conversations and moral support.

\section{Preliminaries} \label{PrelimSection}

In this section we build up some notation and tools needed in the proofs of our theorems.
\subsection{Some group theory}

Let $\Gamma$ be a group. Set $\Gamma^{(k)}$ to be the  \emph{derived series} of $\Gamma$, defined recursively by
$$\Gamma^{(0)} = \Gamma \text{ and } \Gamma^{(k)} = [\Gamma^{(k-1)}, \Gamma^{(k-1)}].$$
A group $\Gamma$ is said to be \emph{solvable} if $G^{(k)} = 1$ for some natural number $k$.
The minimal such $k$ is called the \emph{solvable class} of $\Gamma$.
If, in addition to $\Gamma$ being solvable, each quotient $\Gamma^{(k)}/\Gamma^{(k+1)}$ is finitely generated, then $\Gamma$ is said to be \emph{polycyclic}.
Equivalently, a group $\Gamma$ is polycyclic if and only if $\Gamma$ is solvable and every subgroup of $\Gamma$ is finitely generated.
Set $\Gamma_k$ to be the lower central series for $\Gamma$, defined recursively by
$$\Gamma_0 = \Gamma \text{ and } \Gamma_k = [\Gamma_{k-1}, \Gamma].$$
A group $\Gamma$ is said to be \emph{nilpotent} if $\Gamma_k = 1$ for some natural number $k$. The minimal such $k$ is the \emph{nilpotent class} of $\Gamma$.

We record the following elementary lemma:

\begin{lemma} \label{GroupLemma}
Let $\Gamma$ be a finitely generated group.
Then the following are equivalent:
\begin{enumerate}
\item $\Gamma$ is polycyclic and virtually nilpotent
\item $\Gamma$ is solvable and virtually nilpotent
\end{enumerate}
\end{lemma}

\begin{proof}
The implication (1) $\implies$ (2) is immediate since polycyclic groups are solvable.
We now show that (2) $\implies$ (1).
Since $\Gamma$ is virtually nilpotent, $\Gamma$ contains some nilpotent subgroup, say $\Delta$, of finite-index.
It suffices to show that any subgroup $\Delta'$ of $\Gamma$ is finitely generated.
The group $\Delta'\cap \Delta$ is finite index in $\Delta'$ and is a subgroup of $\Delta$.
Since $\Delta$ is f.g. nilpotent and hence polycyclic, it follows that $\Delta \cap \Delta'$ is finitely generated.
Hence $\Delta'$ is finitely generated as finite generation is inherited by finite group extensions.
\end{proof}

\subsection{Quantifying residual finiteness}

Recall that a group is \emph{residually finite} if the intersection of all its finite index subgroups is trivial.
Let $\Gamma$ be a finitely generated, residually finite group. 
Let $P$ be a property of groups.
For $\gamma \in \Gamma \sm \{1\}$ we define
$Q(\gamma,\Gamma, P)$ to be the set of finite quotients of $\Gamma$ with property $P$ in which the image of $\gamma$ is non-trivial. We say that these quotients \emph{detect} $\gamma$. 
We define
\[ \D^P_\Gamma(\gamma) := \inf\{ |Q| : Q \in Q(\gamma,G, P)\}. \]
For a fixed finite generating set $\mc{X} \subset \Gamma$ we define
\[ \DM^P_{\Gamma,\mc{X}}(n) := \max\{ \D_\Gamma^P(\gamma) \;:\;\gamma \in \Gamma \sm \{ 1 \}, \| \gamma \|_\mc{X} \leq n \}.\]
For two functions $f,g : \N \rw \N$ we write $f \preceq g$ if there exists a natural number $M$ such that $f(n) \leq Mg(Mn)$, and we write $f \approx g$ if $f \preceq g$ and $g \preceq f$. 
In the case when $f \approx g$ does not hold we write $f \not \approx g$.
When $f \preceq g$ does not hold we write $f \not \preceq g$.
We will also write $f \succeq g$ for $g \preceq f$ and $f \not \succeq g$ for $g \not \preceq f$.
If there exists a natural number $M$ such that $f(n) \leq Mg (Mn)$ for infinitely many $n$, we say that \emph{$f(n) \preceq g(n)$ for infinitely many $n$}.

It was shown in \cite{B10} that if $\mc{X},\mc{Y}$ are two finite generating sets for the residually finite group $\Gamma$, then $\DM_{\Gamma,\mc{X}} \approx \DM_{\Gamma,\mc{Y}}$. 
This result actually holds for the more general $\DM^P_\Gamma$ function when the property $P$ is always inherited by subgroups:

\begin{lemma}\label{DivisibilityAsymptoticLemma}
Let $\Gamma$ be a finitely generated group and $P$ be a property of groups that is always inherited by subgroups.
If $\Delta$ is a finitely generated subgroup of $\Gamma$ and $\mc{X},\mc{Y}$ are finite generating sets for $\Gamma,\Delta$ respectively, then $\DM^P_{\Delta,\mc{Y}} \preceq \DM^P_{\Gamma,\mc{X}}$.
\end{lemma}

\begin{proof}
As any homomorphism of $\Gamma$ to $Q$, with $Q$ having property $P$, restricts to a homomorphism of $\Delta$ to a subgroup of $Q$, it follows that $D_\Delta^P(h) \leq D_\Gamma^P(h)$ for all $h \in \Delta$.
Hence,
\begin{equation} \label{firstlemmaeq1}
\normF_{\Delta,\mc{Y}}^P(n) = \sup\{  \D_\Delta^P (g) \: :\: \| g \|_\mc{Y} \leq n , g \neq 1\} \leq 
\sup\{  \D^P_\Gamma (g) \:: \: \| g \|_\mc{Y} \leq n, g \neq 1 \}.
\end{equation}

\noindent
Further, there exists a $C>0$ such that any element in $\mc{Y}$ can be written in terms of at most $C$ elements of $\mc{X}$.
Thus,

\begin{equation} \label{firstlemmaeq2}
\{ h \in \Delta : \| h \|_\mc{Y} \leq n , h \neq 1\} \subseteq \{ g \in \Gamma : \| g \|_\mc{X} \leq Cn, g \neq 1 \}.
\end{equation}

\noindent
So by (\ref{firstlemmaeq1}) and (\ref{firstlemmaeq2}), we have that
$$
\normF_{\Delta,\mc{Y}}^P(n) \leq 
\sup\{  \D^P_\Gamma (g) : \| g \|_\mc{Y} \leq n, g \neq 1 \}
\leq 
\sup\{  \D^P_\Gamma (g) :  \| g \|_\mc{X} \leq C n, g \neq 1 \}
= \normF_{\Gamma,\mc{X}}^P(C n),$$
as desired.
\end{proof}

\noindent
Since we will only be interested in asymptotic behavior, we let $\DM_\Gamma^P$ be the equivalence class (with respect to $\approx$) of the functions $\DM_{\Gamma,\mc{X}}^P$ for all possible finite generating sets $\mc{X}$ of $\Gamma$. Sometimes, by abuse of notation, $\DM_\Gamma^P$ will stand for some particular representative of this equivalence class, constructed with respect to a convenient generating set.

\subsection{Connections to word growth and normal subgroup growth} \label{SolvableGrowthSection}

Given a finitely generated group $\Gamma$ with generating set $\mc{X}$, recall that \emph{word growth} involves studying the asymptotics of the following function:
\begin{eqnarray*}
w_\Gamma(n) := | \{ \gamma \in \Gamma : \text{ the word length of $\gamma$ with respect to $\mc{X}$ is no more than $n$} \} |.
\end{eqnarray*}
\emph{Subgroup growth} is the asymptotic growth of 
\begin{eqnarray*}
s_\Gamma(n) := | \{ \Delta \leq \Gamma : [\Gamma: \Delta] = n \} |.
\end{eqnarray*}
\emph{Normal subgroup growth} is the asymptotic growth of
\begin{eqnarray*}
s^\lhd_\Gamma(n) := | \{ \Delta \lhd \Gamma : [\Gamma: \Delta] = n \} |.
\end{eqnarray*}

\noindent
Gromov's Polynomial Growth theorem \cite{G81} equates virtual nilpotency to having polynomial word growth. The following lemma is a slight improvement of Proposition 2.3 in \cite{BM1}.

\begin{lemma} \label{GrowthTheorem}
Let $\Gamma$ be a finitely generated, residually finite group generated by $\mc{X}$. If
$$\exp(\exp([\log\log(n)]^{3})) \preceq w_{\Gamma,\mc{X}}(n),$$
for infinitely many $n$, then $\normF_{\Gamma,X}(n) \not \preceq (\log(n))^r$ for any $r \in \R$.
\end{lemma}

\begin{proof}
We first recall a basic inequality from \cite{BM1} (Inequality (1) in that paper) that relates word growth, normal subgroup growth, and normal residual finiteness growth:
\begin{equation} \label{BasicInequality}
 \log (\word_{\Gamma,\mc{X}}(n)) \leq \sub^\lhd_\Gamma(\DM_{\Gamma,\mc{X}}(2n))\log (\DM_{\Gamma,\mc{X}}(2n)).
 \end{equation}

To prove the theorem with this inequality, assume to the contrary that there exists $r \in \R$ such that $\DM_{\Gamma,\mc{X}} \preceq (\log(n))^r$. In terms of $\preceq$ notation, inequality (\ref{BasicInequality}) becomes:
\[ \log(\word_{\Gamma, \mc{X}} (n)) \preceq \sub^\lhd_\Gamma (\DM_{\Gamma, \mc{X}}(n)) \log(\DM_{\Gamma, \mc{X}}(n)). \]
Taking the log of both sides, we obtain
\[ \log\log(\word_{\Gamma, \mc{X}} (n)) \preceq \log(\sub^\lhd_\Gamma (\DM_{\Gamma, \mc{X}}(n)))+  \log(\log(\DM_{\Gamma, \mc{X}}(n))). \]
This inequality, in tandem with the assumptions
\begin{align*}
\exp(\exp([\log\log(n)]^{3}))&\preceq \word_{\Gamma,\mc{X}}(n) \text{ for infinitely many $n$}, \\
\DM_{\Gamma,\mc{X}}(n) &\preceq (\log(n))^r,
\end{align*}
and $\log(\sub^\lhd_\Gamma(n)) \preceq (\log(n))^2$ (see \cite[Corollary 2.8]{LS03}) gives 
$$[\log\log(n)]^{3} \preceq (\log\log(n))^2 + \log\log\log(n),$$
for infinitely many $n$, which is impossible.

\end{proof}

\noindent
Following the proofs  in \cite{BM1}, we achieve the following two corollaries.
\begin{corollary} \label{MainCorollaryA}
Any finitely generated solvable group has normal residual finiteness growth which is at most polynomial in $\log(n)$ if and only if the group is virtually nilpotent. \qed
\end{corollary}

\begin{proof}
Since virtually nilpotent groups are linear \cite{A67}, any virtually nilpotent group has polynomial in $\log(n)$ normal residual finiteness growth (see \cite{B10}).
If suffices to show that any finitely generated solvable group that has normal residual finiteness growth which is at most polynomial in $\log(n)$ is virtually nilpotent.
Milnor's Theorem in \cite{M68} states that any finitely generated solvable group which is not virtually nilpotent must have exponential word growth.
This fact along with the normal residual finiteness growth assumption on $\Gamma$ contradicts Lemma \ref{GrowthTheorem}.
\end{proof}

\begin{corollary} \label{MainCorollary}
Any finitely generated solvable group $\Gamma$ that is virtually nilpotent has solvable residual finiteness growth bounded above by $(\log(n))^k$ for some $k$. \qed
\end{corollary}

\section{The Proofs of Theorems \ref{MainTheorem} and \ref{SecondaryTheorem}} \label{ProofSection}

\begin{proof}[Proof of Theorem \ref{MainTheorem}]
It follows, from Theorem 0.2 in \cite{B10}, that if $\Gamma$ is virtually nilpotent then $\DM_\Gamma(n)$ is at most polynomial in $(\log(n))$.
Hence, it suffices to show that $\Gamma$ is virtually nilpotent if $\DM_\Gamma(n)$ is at most polynomial in $\log(n)$.
Let $\Delta$ be a finite-index subgroup of $\Gamma$ that is residually solvable and finitely generated.
If $\Delta$ is virtually nilpotent, then $\Gamma$ is virtually nilpotent.
Further, by Lemma \ref{DivisibilityAsymptoticLemma}, $\DM_\Delta(n)$ is at most polynomial in $\log(n)$.
Hence, we may assume that $\Gamma$ is residually solvable and finitely generated.
Suppose $\Gamma$ is not virtually nilpotent, then by Theorem 1.1 in \cite{Wi05}, it follows that $\Gamma$ must have word length greater than
$$
\exp\exp([\log(n)]^{1/3})
$$
for infinitely many $n$.
But we claim that having such word growth contradicts Lemma \ref{GrowthTheorem}.
Indeed, the assumption in Lemma \ref{GrowthTheorem} is satisfied if
$$
\exp(\exp([\log\log(n)]^3)) \preceq \exp(\exp([\log(n)]^{1/3})),
$$
which is clearly true. 
\end{proof}

\begin{proof}[Proof of Theorem \ref{SecondaryTheorem}]
If $\Gamma$ is virtually nilpotent and residually solvable, then by Lemma 0.4 in \cite{LM91}, $\Gamma$ is solvable.
So by Corollary \ref{MainCorollary}, we see that $\Gamma$ must have solvable residual finiteness growth which is at most polynomial in $\log(n)$.
This completes one direction of the proof.
To finish, we must show that if $\Gamma$ has solvable residual finiteness growth which is at most polynomial in $\log(n)$, then $\Gamma$ is virtually nilpotent and solvable.
By Theorem \ref{MainTheorem}, $\Gamma$ is virtually nilpotent.
Hence, as $\Gamma$ is residually solvable, $\Gamma$ must be solvable by Lemma 0.4 in \cite{LM91}.
\end{proof}

\section{Proof of Theorem \ref{freegrouptheorem}} \label{freegroupsection}

Let $\Gamma$ be the fundamental group of a compact surface.
Since $\Gamma$ contains a free group, the lower bound of the theorem follows from the following claim and Lemma \ref{DivisibilityAsymptoticLemma}.

\begin{claim}
We have $\DM_{\Z*\Z}^{\nil} (n) \succeq \exp[ \sqrt{n} ]$.
\end{claim}

\begin{proof}
Let $x$ and $y$ be generators for $\Z*\Z$.
Let $u_1(x,y) = x^{-2} y^{-1} x$ and 
$$u_{n+1}(x,y) = [ x u_n(x,y) x^{-1}, y u_n(x,y) y^{-1} ].$$
By Theorem 1.1 in \cite{B06}, $u_n(x,y) \neq 1$ for all $n$.
Moreover, $u_n \in \Gamma^{(n)}$ and $\| u_n \| \leq C 4^n$ for some natural number $C$.
By a well-known result of Hall, $\Gamma^{(n)} \subset \Gamma_{2^n}$ (see, for example, Lemma 2.7 in \cite{MP}).
Further recall that if $Q$ is a finite group of nilpotence class $2^n$, then $|Q| > 2^{2^n}$.
Drawing all this together gives
$$
\DM_\Gamma^{\nil} (C 4^n) \geq 2^{2^n}.
$$
Let $m = [\log_2(\sqrt{n})]$, where $[\cdot]$ denote the floor function and $n$ is large enough to ensure that $m$ is positive.
We have
$$
\DM_\Gamma^{\nil} (C 4^m) \geq 2^{2^m}.
$$
Since $[\log_2(\sqrt{n})] \geq \log_2(\sqrt{n})$, we have $2^{2^m} \geq 2^{\sqrt{n}}$.
Further, $[\log_2(\sqrt{n})] < \log_2(\sqrt{n}) + 1$, so because $F_\Gamma^{\nil}$ is a nondecreasing function in $n$,
$$
\DM_\Gamma^{\nil} (C2^{2[\log(\sqrt{n})]})
\leq \DM_\Gamma^{\nil}(C2^{2\log_2(\sqrt{n})+2})
\leq \DM_\Gamma^{\nil}(4 C n).
$$
Hence,
$$
\DM_\Gamma^{\nil} (4 C n) < 2^{\sqrt{n}}.
$$
Since $4 C$ is greater than one, we get
$$
\DM_\Gamma^{\nil} (n) \preceq 2^{\sqrt{n}}.
$$
\end{proof}

Before proving the upper bound in Theorem \ref{freegrouptheorem}, we first prove some preliminary results.

\begin{claim} \label{TwoGroupClaim}
Let $\Delta$ be the kernel of the group homomorphism 
$$\phi: \PSL_2(\Z[i]) \to \SL_2(\Z/2\Z)$$
 induced by the ring homomorphism 
$$\Z[i] \to \Z[i] / \left< (1-i) \right> = \Z/2\Z.$$
Then $\DM_\Delta^{\nil} (n) \preceq 2^n$.
\end{claim}

\begin{proof}
Before starting the proof, we construct a filtration for the kernel which will help us find small nilpotent quotients.
Set $G_0 = \Delta$ and 
$$
G_k = \ker[ \PSL_2(\Z[i]) \to \SL_2((\Z/2^k\Z)[i])/\{\pm 1\}]\text{ for } k \geq 1.
$$
Then, because $2 = (1-i)(1+i)$, we have the following filtration of normal subgroups for $G_0$:
$$
G_0 \geq G_1 \geq G_2 \geq G_3 \geq \cdots
$$
We first claim that each quotient $G_0/G_{k}$ is a $2$-group of order less than $2^{8k}$.
We write $[A]$ for the equivalence class in $\PSL_2(\Z[i])$ containing an element $A \in \SL_2(\Z[i])$.
The first quotient $G_0/G_1$ is 
$$\{ [A] \in \PSL_2( \Z[i] ) : A \equiv 1 \;\mod\; (1-i) \Z[i] \}/ \sim,$$
where $[A] \sim [B]$ if $A \equiv \pm B \equiv B \; \mod \; 2 \Z[i]$.
Denote by $M_2(\Z/2\Z[i])$ the set of all $2\times 2$ matrices with coefficients in $\Z/2\Z[i]$.
Let $h : G_0/G_1 \to (M_2(\Z/2\Z[i]), +)$ be given by $[A] G_1 \mapsto (A-1)$.
The map is well-defined:
indeed if $[A] \sim [B]$ then $[A] = [B + 2N]$ for some $N$ in $M_2(\Z[i])$. So setting $M_A = A-1$ gives $$[B] = [A + 2N] =[1 + M_A + 2N].$$
Hence, $h([A]G_1) = M_A$ and $h([B] G_1) = M_A + 2N$ which is equal to $M_A$ in $M_2(\Z[i]/2\Z[i])$.
Further, the map is a homomorphism:
indeed, if $[A] = [1 + (1-i) M]$ and $[B] = [1 + (1-i) N]$, then
$$
[A][B] = [1 + (1-i)(M +N) - 2i MN] \sim [1 + (1-i) (M + N)].
$$
Finally, the map $h$ is injective since matrices that get mapped to $1$ under $h$ must satisfy $A \equiv 1 \;\mod \;2 \Z[i]$. It follows that $G_0/G_1$ is a $2$-group of order at most $|M_2(\Z/2\Z[i])| = 2^{8}$.

For $k > 1$ we write
$$G_k /G_{k+1} = \{ [A] \in \PSL_2( \Z[i] ) : A \equiv 1 \;\mod\; 2^{k} \Z[i] \}/ \sim,$$
where $[A] \sim [B]$ if $A \equiv \pm B\; \mod\; 2^{k+1} \Z[i]$.
Define a map 
$$g: G_k/G_{k+1} \to (M_2(\Z/2\Z[i]), +)$$
 by $[A] \mapsto (A-1)2^{-k}$.
It is well-defined: indeed if $[A] \sim [B]$ then $A= \pm(B - 2^{k+1} N)$ for some $N$ in $M_2(\Z[i])$. So setting $M_A = (A-1)2^{-k}$ gives $$\pm B = A \pm 2^{k+1} N = 2^k M_A + 1 \pm 2^{k+1} N = 1 + 2^k( M_A \pm 2N).$$
Hence, $h([A] G_{k+1}) = M_A$ and $h([B] G_{k+1}) = M_A + 2N$ which is equal to $M_A$ in $M_2(\Z[i]/2\Z[i])$.

Further, $g$ is a homomorphism:
Indeed, if $[A] = [1 + 2^k M]$ and $[B] = [1 + 2^k N]$, then $[A][B] = [AB]$ becomes
$$
[(1 + 2^k M) (1 + 2^k N )] = [1 + 2^k M + 2^k N + 2^{2k} M N] \sim [1 + 2^k( M + N)],
$$
which maps to $M + N.$
Finally, the map is injective, as $g([A]G_{k+1}) = 0$ implies that $[A] = [1 + 2^{k+1} N] \sim [1]$.
It follows that $G_k/G_{k+1}$ is a two group with order bounded above by $2^{8}$. This gives that $G_0/G_k$ is a $2$-group of order bounded above by $2^{8k}$, as claimed.

Let $\mc{X}$ be a finite set of generators for $\Delta$ as a semigroup and set $S  = \{ 1, i \}$.
We claim that there exists $\lambda>0$ such that for any $[A] \in \Delta$ with $\|[A]\|_{\mc{X}} \leq n$ and any non-zero entry $a \in \Z[i]$ of $A\pm1$ we have
\[ \|a\|_S \leq \lambda^n. \]

To prove the claim, let $a'$ be the entry of $A$ in the same spot as $a$ in $A\pm1$.
We have, by the triangle inequality,
\[ \|a\|_S \leq  \|a'\|_S + \|1\|_S. \]
This reduces the above claim to the following. There exists $\mu>0$ such that for any $A \in G_0$ with $\|A\|_\mc{X}\leq n$ and any non-zero entry $a \in \Z[i]$ of $A$ we have
\[ \|a\|_S \leq \mu^n. \]
We claim that if $\beta$ denotes the maximum of $\|x\|_S$, where $x$ ranges over all entries of all elements of $\mc{X}$, then $\mu := m\beta$ satisfies the last statement. To see this, consider first the case $A=XY$ with $X,Y \in \mc{X}$. The entries of $A$ are scalar products of the rows of $X$ and the columns of $Y$. Thus we are led to study $\|x\cdot y\|_S$ for $x,y \in \Z[i]^m$, where $\cdot$ denotes scalar product. Clearly we have $\|x\cdot y\|_S \leq m\max\{\|x_j y_j\|_S:1 \leq i \leq m\}$. In terms of the basis $S$ we can write
\[ x_j = a_{x}  + b_{x} i \qquad \tx{and}\qquad y_j = a_y + b_y i \]
where each $a_x, b_x, a_y, b_y$ belongs to $\Z$. One computes
\[ \|x_jy_j\|_S \leq \|x_j\|_S\|y_j\|_S. \]
This formula and induction on $n$ complete the proof of the claim.

We now finish the proof. Let $[A]$ be an element in $\Delta$ of word length at most $n$ in terms of $\mc{X}$.
Then, by the above, there exists a $C > 0$ such that any nonzero entry $a$ of $A-1$ or $A+1$ satisfies
$$
\|a \|_S \leq 2^{Cn}.
$$
However, by the definition of $G_k$: if $A$ is nontrivial and in $G_k$, then any nonzero entry $a$ of $A-1$ and $A+1$ has $\|a\|_S \geq 2^k $.
It follows that $[A] \notin G_{Cn}$, hence, as $G_0/G_{Cn}$ is a $2$-group of order at most $2^{8 Cn}$, we have
$$
\D^{\nil}([A]) \leq 2^{8 C n},
$$giving
$$
\DM_{G_0}^{\nil} (n) \preceq 2^n.
$$
\end{proof}

\begin{claim} \label{ReidClaim}
Let $\Delta$ be the kernel of the group homomorphism 
$$\phi: \PSL_2(\Z[i]) \to \SL_2(\Z/2\Z)$$
 induced by the ring homomorphism 
$$\Z[i] \to \Z[i] / \left< (1-i) \right> = \Z/2\Z.$$
Then $\Delta$ contains the fundamental group of a genus 2 surface.
\end{claim}

\begin{proof}
Let $d$ be a square-free postive integer, $\mathcal{O}_d$ the ring of integers
of the quadratic imaginary number field $\Q(\sqrt{-d})$ and
$\Gamma_d$ the Bianchi group $\PSL_2(\mathcal{O}_d)$.
It was shown by Maclachlan (see Chapter 9.6 of \cite{MR}) that for any circle
$\mathcal C$ with equation:

$$a|z|^2 + Bz + \overline{B}\overline{z} + c = 0,~~\hbox{with}~~a,c\in \Z, B\in \mathcal{O}_d$$

\noindent
the group

$$\Stab({\mathcal C},\Gamma_d)=\{\gamma\in\Gamma_d:\gamma({\mathcal C}) ={\mathcal C}~~
\hbox{and}~~\gamma~~\hbox{preserves components of}~~{\bf C}\setminus {\mathcal C}\}$$
\noindent
is an arithmetic Fuchsian subgroup of $\Gamma_d$.  Moreover, all such arithmetic
Fuchsian subgroups occur like this.

We now fix attention on the case of $d=1$ and the circle ${\mathcal C}$ with equation:

$$2|z|^2 + (1+i)z + (1-i)\overline{z} - 2 = 0.$$
\noindent
Denote the group $\Stab({\mathcal C})$ by $F$.  It is shown in \cite{MR1} Theorem
8 that this is an arithmetic Fuchsian group of signature $(0;2,2,3,3;0)$.
With a bit of effort, explicit generators for $F$ can be computed, namely:
$$\tiny{x_1= \begin{pmatrix} i & 1+i \cr 0 & -i \end{pmatrix}, \:
x_2=\begin{pmatrix} -1+2i & 3+i \cr 1+i & 1-2i \end{pmatrix},\:
x_3= \begin{pmatrix} 2i & 3+2i \cr 1 & 1-2i \end{pmatrix}, \:
x_4= \begin{pmatrix} 1+2i & 2+3i \cr -i & -2i \end{pmatrix}.}$$
\noindent
Now let $\Gamma$ denote the principal congruence subgroup of $\Gamma_1$
of level $\left<(1+i) \right>$.  To determine the group $F\cap \Gamma$ we consider the reduction
of the generators of $F$ modulo $\left<(1+i)\right>$.  It is easily seen that 
(projectively) $x_1$ and $x_2$ map trivially with $x_3$ and $x_4$ 
mapping to the elements
$$\begin{pmatrix} 0 & 1 \cr 1 & 1 \end{pmatrix}~~\hbox{and}~~\begin{pmatrix} 1 & 1 \cr 1 & 0 \end{pmatrix},$$
\noindent
respectively.  We deduce from this that the image of $F$ under the
reduction homomorphism is cyclic of order $3$, and so it follows
that since each of $x_1$ and $x_2$ is killed, they
determine 3 conjugacy classes
of cyclic groups of order $2$ in $F\cap \Gamma$.  
Given that the index $[F:F\cap \Gamma]=3$,
we deduce that 
the signature of the Fuchsian group $F\cap \Gamma$
is $(0;2,2,2,2,2,2;0)$.  Any such group has a genus $2$ surface group.  In
summary we have shown that the level $\left<(1+i) \right>$ principal congruence subgroup of $\PSL_2(\mathcal{O}_1)$
contains a genus $2$ surface group.
\end{proof}

We now prove the upper bound in Theorem \ref{freegrouptheorem}:

\begin{corollary}
Let $\Gamma$ be the fundamental group of an oriented surface.
Then $\DM^{\nil}_\Gamma(n) \preceq 2^n$.
\end{corollary}

\begin{proof}
By Claim \ref{ReidClaim}, the kernel in Claim \ref{TwoGroupClaim} contains $\Gamma$.
Hence the corollary follows from Lemma \ref{DivisibilityAsymptoticLemma}.
\end{proof}

\begin{corollary} \label{MainCorollary2}
Let $\Gamma$ be the fundamental group of an oriented surface with generating set $\mc{X}$.
Then there exists a constant $C > 0$ such that any $\gamma \in \Gamma$ satisfies $\gamma \notin \Gamma_{C\| \gamma\|_\mc{X}}$.
\end{corollary}

\begin{proof}
By the previous corollary, we have that there exists $C > 0$ such that $\DM^{\nil}_{\Gamma,\mc{X}}(n) < 2^{C n}$ for all $n$.
Hence, for any $\gamma \in \Gamma \sm \{1\}$, we have $\D^{\nil}_\Gamma(\gamma) \leq 2^{C  \| \gamma \|_\mc{X} }$.
But any finite group of nilpotent class no less than $C \| \gamma \|_\mc{X}$, must have order no less than $2^{C \| \gamma \|_\mc{X}}$, hence $\gamma \notin \Gamma_{C \| \gamma \|_\mc{X}+1}$, which is sufficient.
\end{proof}

\section{Some examples} \label{ExampleSection}

In this section we show that better lower bounds can be found for some groups $\Gamma$ where $[\Gamma,\Gamma]$ is not finitely generated.

\begin{example} \label{ex:example1}
Let $p$ be a prime number.
Let $\Gamma$ be the group $\Z/p\Z \wr \Z$. Set $\Delta = \oplus_{i \in \Z} \Z/p\Z$ to be the base group of $\Gamma$ so $\Gamma/\Delta \cong \Z$.
We claim that $\DM_\Gamma(n) \succeq \sqrt{n}$.
\end{example}

\begin{proof}
We begin with a linear algebraic construction of candidates for a lower bound.
Let $A$ be the $2m \times m$ matrix with entries in $\Z/p\Z$ given by
\begin{equation} \label{matrixA}\tiny{
\begin{pmatrix}
1 & 1 & 1 & 1 & 1 & \cdots  \\
1 & 0 & 1 & 0 & 1 & \cdots  \\
0 & 1 & 0 & 1 & 0 & \cdots  \\
1 & 0 & 0 & 1 & 0 & \cdots  \\
0 & 1 & 0 & 0 & 1 & \cdots  \\
0 & 0 & 1 & 0 & 0 & \cdots \\
1 & 0 & 0 & 0 & 1 & \cdots \\
\vdots &  & & & &
\end{pmatrix}.}
\end{equation}
Set $m = n(n+1)/2$.
This matrix gives a linear transformation from $(\Z/p\Z)^{2m}$ to $(\Z/p\Z)^{m}$.
Since the cardinality of $(\Z/p\Z)^{2m}$ is greater than that of $(\Z/p\Z)^m$, we have that $\ker(A)$ is nontrivial. 
Fix $w = (w_1, \ldots, w_{2m})$ to be a nontrivial element in $\ker(A)$.
Let $\delta_i$ be the Dirac delta function giving an element in $\Delta$.
Set $v$ to be the element 
$$
\sum_{i=1}^{2m} w_i \delta_i.
$$
This element $v$ is our candidate.

Let $t$ be the generator for $\Z$ in the wreath product of $\Z/p\Z \wr \Z$.
It is straightforward to see that the element $v$ has word length less than $2m(p+2)$.
Let $\phi: \Gamma \to Q$ be a map onto a finite $Q$ such that $v \notin \ker \phi$.
Then, supposing that $\phi(t^r) = 1$ for $r < n$ gives
$$
\phi(v) = \phi \left( \sum_{i=1}^{2m} w_i \delta_i  \right)
= \phi \left( \sum_{i=1}^{r-1} (w_i + w_{i+r} + w_{i+2r} + \cdots )  \delta_i \right)= 0.
$$
Hence, $r \geq n$, and so $|Q| \geq n$, giving $\normF_\Gamma(n^2) \succeq n$ and hence $\normF_\Gamma(n) \succeq n^{1/2}$ as desired.
\end{proof}

\noindent
The above method strengthens Example 2.3 in \cite{BK10}.

\begin{example} \label{ex:example2}
Let $\Gamma$ be the wreath product $\Z \wr \Z$. Set $\Delta = \oplus_{i \in \Z} \Z$ to be the base group of $\Gamma$ so $\Gamma/\Delta \cong \Z$.
We claim that $\DM_\Gamma(n) \succeq n^{1/4}$.
\end{example}

\begin{proof}
We begin with a linear algebraic construction of candidates for a lower bound.
As in the above proof, let $A$ be the $2m \times m$ matrix, this time with entries in $\Z$ given by
Expression \ref{matrixA}.
Set $m = n(n+1)/2$, for $n$ even.
This matrix gives a map from $V:= \Z^{2m}$ to $W:=\Z^{m}$.
Let $\| \cdot \|_{\Z^\ell}$ be the supremum norm on $\Z^\ell$.
By the triangle inequality and the definition of $A$, one sees that for any $v \in V$, we have
$$
\| Av \|_W \leq m \| v \| _V.
$$
Further, if we let $B_{\Z^\ell}(k) = \{ v \in \Z^\ell : \| v \|_{\Z^\ell} \leq k \}$, then $AB_V(k) \subseteq B_W(mk)$ for all $k$.
Moreover, for all even $k \in \N$,
\begin{eqnarray*}
 |B_W(k/2)| &=& (k+1)^m, \text{ and } \\
|B_V(k/2)| &=& (k+1)^{2m}.
\end{eqnarray*}
\noindent
If $k = (m+2)$, then 
$$|B_V(k/2)| = (k+1)^{2m} > ((m+3)k)^m = (km + 3k)^m > (mk+1)^m = |B_W(mk/2)|.$$
Hence, as $AB_V(k/2) \subseteq B_W(mk/2)$, there must exist an element $w \in \ker(A) \cap B_V(k)$ that is nontrivial.
Fix such a $w$, and write $w = (w_1, \ldots, w_{2m})$.
Let $\delta_i$ be the Dirac delta function giving an element in $\Delta$.
Set $v$ to be the element 
$$
\sum_{i=1}^{2m} w_i \delta_i.
$$
This element $v$ is our candidate.
It is straightforward to show that the element $v$ has word length less than $2m(k+2)$.

Let $t$ be the canonical generator for the wreath product of $\Z \wr \Z$ whose image generates $\Gamma/\Delta$.
Let $\phi: \Gamma \to Q$ be a map onto a finite $Q$ such that $v \notin \ker \phi$.
Then, supposing that $\phi(t^r) = 1$ for $r < n$ gives
$$
\phi(v) 
= \phi \left( \sum_{i=1}^{2m} w_i \delta_i \right)
= \phi \left( \sum_{i=1}^{r-1} (w_i + w_{i+r} + w_{i+2r} + \cdots ) \delta_i \right)= 0.
$$
Hence, $r \geq n$, and so $|Q| \geq n$, giving $\normF_\Gamma(n^4) \succeq n$ and hence $\normF_\Gamma(n) \succeq n^{1/4}$ as desired.
\end{proof}

%---------------------------------------------------------------------------

\end{document}